\documentclass[12pt]{amsart}
\usepackage{amssymb}
\usepackage{amsmath,amssymb,amsfonts,amsthm,graphics,
latexsym, amscd, amsfonts, epsfig, eepic,epic,psfrag,setspace}
\usepackage{mathrsfs}
\usepackage{color}
\usepackage{eucal}%    caligraphic-euler fonts: \mathcal{ }
\usepackage{xspace}
\usepackage{amsthm}
\usepackage{mathrsfs}

\theoremstyle{plain}
\newtheorem{theorem}{Theorem}
\newtheorem{corollary}{Corollary}

\newtheorem{lemma}{Lemma}
\newtheorem{proposition}{Proposition}

\theoremstyle{definition}

\theoremstyle{remark}

\numberwithin{equation}{section}

\usepackage{natbib}

\begin{document}

\doublespacing
%%%
%%%
%%%%%%%%%%%%%%%%%%%%%%%%%%%%%%%%%%%%%%%%%%%%%%%%%%%%%%%%%%%%%%%%%%%%%%%%%%
%%
%%%
\begin{center}
{\bf\Large The 5'-3' distance of RNA secondary structures}
\\
\vspace{15pt} Hillary S.~W. Han and Christian M. Reidys$^{\,\star}$
\end{center}

\begin{center}
%XXX%
         Institut for Matematik og Datalogi  \\
         University of Southern Denmark \\
         Denmark\\
         Phone: *45-24409251 \\
         Fax: *45-65502325 \\
         email: duck@santafe.edu
\end{center}

%%\centerline{\bf Abstract}

\centerline{\bf Abstract}

\qquad \: Recently Yoffe {\it et al.} observed that the average distances
between $5'$-$3'$ ends of RNA molecules are very small and largely independent
of sequence length. This observation is based on numerical computations as well
as theoretical arguments maximizing certain entropy functionals.
In this paper we compute the exact distribution of $5'$-$3'$ distances of
RNA secondary structures for any finite $n$. We furthermore compute the limit
distribution and show that already for $n=30$ the exact distribution and the
limit distribution are very close. Our results show that the distances of
random RNA secondary structures are distinctively lower than those of minimum
free energy structures of random RNA sequences.

{\bf Keywords}:
RNA secondary structure, singularity analysis, noncrossing
diagram, distance.

%%%
%%%%%%%%%%%%%%%%%%%%%%%%%%%%%%%%%%%%%%%%%%%%%%%%%%%%%%%%%%%%%%%%%%%%%%%%%
%%%

\section{Introduction and background}

%%%
%%%%%%%%%%%%%%%%%%%%%%%%%%%%%%%%%%%%%%%%%%%%%%%%%%%%%%%%%%%%%%%%%%%%%%%%
The closeness of $5'$ and $3'$ ends of RNA molecules has distinct
biological significance, for instance for the replication efficiency
of single stranded RNA viruses or the efficient translation of
messenger RNA molecules. It is speculated in \citep{Yoffe} that this
effective circularization of large RNA molecules is rather a generic
phenomenon of large RNA molecules and independent of sequence length.
It is to large extend attributed to the high number of paired bases.

In this paper we study the distribution of $5'$-$3'$ distances in RNA secondary
structures. We first compute the distribution of $5'$-$3'$ distances
of RNA secondary structures of length $n$ by means of a bivariate
generating function. The key idea is to view secondary
structures as tableaux sequences and to relate the $5'$-$3'$ distance to the
nontrivial returns \citep{Emma:decom} of the corresponding path of shapes.
Secondly, we derive the limit distribution of $5'$-$3'$ distances. The idea is to
compute the singular expansion of the above generating function via
the subcritical paradigm \citep{Flajolet:07a} and to employ a discrete
limit theorem.

Our results prove, that the $5'$-$3'$ distances of random RNA structures are
distinctively smaller than those of biological RNA molecules and
minimum free energy (mfe) RNA structures.
This comes as a surprise since the number of paired bases in random
structures is $55.2 \%$ \citep{Reidys:book} and therefore smaller than
the $60\%$ of mfe structures \citep{Schuster:93}.

An RNA structure is the helical configuration of its primary
sequence, i.e.~the sequence of nucleotides {\bf A}, {\bf G}, {\bf U}
and {\bf C}, together with Watson-Crick ({\bf A-U}, {\bf G-C}) and
({\bf U-G}) base pairs. The combinatorics of RNA secondary structures
has been pioneered by Waterman
\citep{Penner:93c,Waterman:78a,Waterman:79a,Waterman:80,Waterman:94a}.
We interpret an RNA secondary structure as a diagram, i.e.~labeled
graphs over the vertex set $[n]=\{1, \dots, n\}$, represented by
drawing its vertices $1,\dots,n$ in a horizontal line and connecting
them via the set of backbone-edges $\{(i,i+1)'\mid 1\le i\le n-1\}$.
Besides its backbone edges a diagram exhibits arcs, $(i,j)$, that are
drawn in the upper half-plane. Note that an arc of the form $(i,i+1)$
or $1$-arc, is distinguished from the backbone edge $(i,i+1)'$. However,
no confusion can arise since an RNA secondary structure is a diagram
having no $1$-arcs and only noncrossing arcs in the upper half-plane, see
Fig.\ \ref{F:secon}.\\

The $5'$-$3'$ distance of an RNA secondary structure is the minimal length of a path
of the diagram. Such a diagram-path is comprised of arcs and backbone-edges,
see Fig.\ \ref{F:distance0}.
%%%

The paper is organized as follows:
In Section~\ref{S:prelim} we discuss some basic facts, in particular
the structure-tableaux correspondence and how to express the
$5'$-$3'$ distance via such tableaux-sequences.
In Section~\ref{S:combi} we compute ${\bf W}(z,u)$, the bivariate generating
function of RNA secondary structures of length $n$ having distance $d$.
Section~\ref{S:singular} contains the computation of the singular expansion of
${\bf W}(z,u)$ and in Section~\ref{S:limit} we combine our results
and derive the limit distribution. We finally discuss our results in
Section~\ref{S:discuss}.

%%%
%%%%%%%%%%%%%%%%%%%%%%%%%%%%%%%%%%%%%%%%%%%%%%%%%%%%%%%%%%%%%%%%%%%%%%%%%%
%%%
\section{Preliminaries}\label{S:prelim}
%%%
%%%%%%%%%%%%%%%%%%%%%%%%%%%%%%%%%%%%%%%%%%%%%%%%%%%%%%%%%%%%%%%%%%%%%%%%%%
%%%

Let $\mathscr{S}_n$ denote the set of RNA
secondary structures of length $n$, $\sigma_n$.
All results of this paper easily generalize to the case of diagrams with
noncrossing arcs that contain no arcs of length smaller than $\lambda>1$
and to canonical secondary structures \citep{Reidys:book}, i.e.~structures that
contain no isolated arcs.

The distance of $\sigma_n$, $d_n(\sigma_n)$, is the minimum length of a path
consisting of $\sigma$-arcs and backbone-edges from vertex $1$ (the $5^{'}$ end)
to vertex $n$ (the $3^{'}$-end). That is we have the mapping $
d_n\colon \mathscr{S}_n\longrightarrow \mathbb{N}$.

A sequence of shapes $(\lambda_0, \lambda_1, \ldots, \lambda_n)$ is called a
$1$-tableaux of length $n$, $T_n$, if all shapes contain only one row of
squares and (a) $\lambda_0=\lambda_n =\varnothing$, (b) $\lambda_{i+1}$ is
obtained from $\lambda_i$ by adding a square ($+\Box$), removing a square
($-\Box$) or doing nothing ($\varnothing$) and (c) there exists no sequence of
$(+\Box, -\Box)$-steps. Let $\mathscr{T}_n$ denote the set of all $1$-tableaux
of length $n$.
%%%

We come next to the tableaux interpretation of secondary structures. The
underlying correspondence is an immediate consequence of
\citep{Reidys:vac07,Chen,Reidys:07pseu}. We shall subsequently express
the $5'$-$3'$ distance via $1$-tableaux.
%%%
%%%%%%%%%%%%%%%%%%%%%%%%%%%%%%%%%%%%%%%%%%%%%%%%%%%%%%%%%%%%%%%%%%%%%%%%%%
%%%
\begin{proposition}\citep{Reidys:07pseu}
There exists a bijection between RNA secondary structures and $1$-tableaux:
\begin{equation}
\beta_n\colon \mathscr{S_n}_n\longrightarrow \mathscr{T}_n.
\end{equation}
\end{proposition}
%%%
%%%%%%%%%%%%%%%%%%%%%%%%%%%%%%%%%%%%%%%%%%%%%%%%%%%%%%%%%%%%%%%%%%%%%%%%%%
%%%

%%%
\begin{proof}
Given $\sigma_n$, we consider the sequence $(n,n-1,\dots,1)$ and,
starting with $\varnothing$, do the following:\\
$\bullet$ if $j$ is the endpoint of an arc $(i,j)$, we add one square,\\
$\bullet$ if $j$ is the start point of an arc $(j, s)$, we remove one square,\\
$\bullet$ if $j$ is an isolated point, we do nothing.\\
%%%%%%%%%%%%%%%%%%%%%%%%%%%%%%%%%%%%%%%%%%%%%%%%%%%%%%%%%%%%%%%%%
%%%%%%%%%%%%%%%%%%%%%%%%%%%%%%%%%%%%%%%%%%%%%%%%%%%%%%%%%%%%
This constructs a $1$-tableaux of length $n$ and thus defines the map $\beta_n$.
Conversely, given a $1$-tableau $T_n$, $(\varnothing,\lambda^1,\dots,
\lambda^{n-1},\varnothing)$, reading $\lambda^i\setminus\lambda^{i-1}$
from left to right, at step $i$, we do the following:\\
$\bullet$ for a $+\square$-step at $i$ we insert $i$ into the new square,\\
$\bullet$ for a $\varnothing$-step we do nothing,\\
$\bullet$ for a $-\square$-step at $i$ we extract the entry of the rightmost
          square $j(i)$.
The latter extractions generate the arc-set $\{(i,j(i))\mid i
\;\text{\rm is a $-\square$-step}\}$ that contains by definition of $T_n$
no $1$-arcs. Thus this procedure generates a secondary structure of length
$n$ without $1$-arc, which, by construction, is the inverse of
$\beta_n$ and the proposition follows.
\end{proof}

A secondary structure $\sigma_n$ is irreducible if $\beta(\sigma_n)$ is a sequence
of shapes $(\lambda_0,\dots,\lambda_{n})$ such that $\lambda_j\neq \varnothing$ for
$1\le j<n$. An irreducible substructure of $\sigma_n$ is a subsequence $(\lambda_i,
\dots, \lambda_{i+k})$ such that $\lambda_{i-1}=\varnothing$ and $\lambda_{i+k}=
\varnothing$ and $\lambda_j\neq \varnothing$ for $i\le j<i+k$.
In the following we denote the terminal shapes ($\lambda_{i+k}$) of non-rightmost
irreducibles by $\varnothing^*$ and the terminal shape of the rightmost irreducible
by $\varnothing^{\#}$. Accordingly we distinguish three types of shapes
$\varnothing,\varnothing^*$ and $\varnothing^{\#}$.
We can now express the distance in terms of numbers of $\varnothing^*$ and
$\varnothing$ shapes as follows
\begin{equation}
d_n(\sigma_n) = 2 \, \vert \{\varnothing^*\in \beta(\sigma_n)\}\vert +
\vert \{\varnothing\in \beta(\sigma_n)\}\vert.
\end{equation}

%%%
%%%%%%%%%%%%%%%%%%%%%%%%%%%%%%%%%%%%%%%%%%%%%%%%%%%%%%%%%%%%%%%%%%%%%%
%%%

\section{Combinatorial analysis}\label{S:combi}

%%%
%%%%%%%%%%%%%%%%%%%%%%%%%%%%%%%%%%%%%%%%%%%%%%%%%%%%%%%%%%%%%%%%%%%%%%%%%%%%%%
%%%

Let ${\mathbf w}(n,d)$ denote the number of RNA secondary structures
$\sigma_n$ having distance $d_n$. In the following we shall write $d$
instead of $d_n$ and consider
\begin{equation}
{\mathbf W}(z,u)=\sum_{n\geq 0}\sum_{d \geq 0}{\bf w}(n,d)\,z^n u^d,
\end{equation}
the bivariate generating function of the number of RNA secondary
structure of length $n$ having distance $d$ and set ${\mathbf w}(n)=\sum_{d
\geq 0} {\mathbf w}(n,d)$. Let
${\mathbf S}(z)$ denote the generating function of RNA secondary structures
and ${\bf Irr}$(z) denote the generating function of irreducible secondary
structures (irreducibles).
Let furthermore ${\mathscr S}_n$ denote the set of secondary structures of
length $n$ and ${\mathscr I}_n$ denote the set of irreducible structures of
length $n$.

%%%%%%%%%%%%%%%%%%%%%%%%%%%%%%%%%%%%%%%%%%%%%%%%%%%%%%%%%%%%%%%%%%%%%%%
\begin{theorem}\label{T:exact}
The bivariate generating function of the number of RNA secondary structures
of length $n$ with distance $d$, is given by
\begin{equation}
\begin{split}
&{\mathbf W}(z,u)=
\frac{uz^2 ({\mathbf S}(z)-1)}{(1-zu)^2-(1-zu)(zu)^2
({\mathbf S}(z)-1)}
+\frac{z}{1-zu}.
\end{split}
\end{equation}
\end{theorem}
\begin{proof}
We set ${\mathbf V}(z,u)=z/(1-zu)$ and ${\mathbf U}(z,u)={\mathbf W}(z,u)-
{\mathbf V}(z,u)$.\\
{\it Claim 1:} ${\bf Irr}(z)=z^2\left({\mathbf S}(z)-1 \right)$.\\
To prove Claim 1 we consider the mapping
$
\gamma: {\mathscr I}_n \longrightarrow {\mathscr S}_{n-2}
$,
obtained by removing the shapes $\lambda_1$ and $\lambda_{n-1}$ from
$\beta(\sigma_n)$ and removing the rightmost box from all other shapes
$\lambda_j, 2\leq j \leq n-2$.
Note that for $1=(n-1)$ the tableaux $\beta(\sigma_n)$ corresponds to a
$1$-arc which is impossible.
Hence for an irreducible structure $\lambda_1=\Box$ and $\lambda_{n-1}=\Box$ are
distinct shapes and the induced sequence of shapes
$\mu=(\lambda_0,\lambda_2\setminus \Box,\dots,\lambda_{n-2}\setminus\Box,
\lambda_n)$ is again a $1$-tableaux, i.e.~an element of $\mathscr{S}_{n-2}$,
where $\lambda_j\setminus \Box$ denotes the shape $\lambda_j$ with the
rightmost $\Box$ deleted. Thus $\gamma$ is welldefined.
Given a $1$-tableaux $\tau=(\lambda_0,\dots,\lambda_{n-2})$ we consider the map
\begin{equation}
\gamma^*(\tau)=(\lambda_0,\Box,\lambda_1\sqcup\Box,\dots,\lambda_{n-3}
\sqcup \Box,\Box, \lambda_{n-2})
\end{equation}
where $\lambda_j\sqcup \Box$ denotes the shape $\lambda_j$ with a $\Box$
added, see Fig. ~\ref{F:mapIS}.

%%%
By construction, $\gamma^*\circ \gamma={\rm id}$, whence Claim $1$.
Let us first compute the contribution of secondary structures containing at
least one irreducible.\\
{\it Claim 2:} Suppose $\sigma_n$ has distance $d$, then $(i+1)$ irreducibles
can be arranged in exactly ${d-i \choose i+1}$ ways.\\
Indeed, in view of $d=2\, \vert \{\varnothing^*\in \beta(\sigma_n)\}\vert +
\vert \{\varnothing\in \beta(\sigma_n)\}\vert$,
the distance-contribution of the rightmost irreducible and each isolated point
is one, while the contribution of all remaining $i$ irreducibles equals two.
No two such contributions overlap, whence replacing $d$ by $d-i$ we have
${d-i\choose i+1}$ ways to place the $(i+1)$ irreducibles and Claim $2$ follows.
Accordingly, we obtain for fixed $d$
\begin{equation}
\sum_{n>d} {\mathbf u}(n,d)z^n=\sum_{i \geq 0}{d-i \choose i+1}\,
{\bf Irr}(z)^{i+1} z^{d-2i-1},
\end{equation}
where the indeterminant $z$ corresponds to the isolated points and
${\bf Irr}(z)$ represents the irreducible structures labeled by the
$\varnothing^{*}$ and $\varnothing^{\sharp}$.
Consequently, rearranging terms we derive
\begin{equation}
{\mathbf U}(z,u)=\sum_{d\geq 1}\sum_{n>d}\,{\mathbf u}(n,d)z^n u^d
=\sum_{i \geq 0}\sum_{d \geq 1} {d-i \choose i+1}\,
{\bf Irr}(z)^{i+1} z^{d-2i-1}\, u^d
\end{equation}
and therefore
\begin{equation}
\begin{split}
{\mathbf U}(z,u)
&=\sum_{i \geq 0}\sum_{d \geq 1} {d-i \choose i+1}\, (zu)^{d-i} (z)^{-i-1}\,
u^i\, {\bf Irr}(z)^{i+1} \\
&=
\sum_{i \geq 0}\sum_{d \geq 1} {d-i \choose i+1}\, (zu)^{d-i} \,
\left(\frac{u \,{\bf Irr}(z)}{z}\right)^i \, \frac{{\bf Irr}(z)}{z}.
\end{split}
\end{equation}
Using $\sum_{r \geq 0} {r \choose k}\,x^r= \frac{x^k}{(1-x)^{k+1}}, k \geq 0$,
we compute
\begin{equation*}
\begin{split}
{\mathbf U}(z,u)&=\sum_{i \geq 0}\,\frac{(zu)^{i+1}}{(1-zu)^{i+2}}\,
\left(\frac{u {\bf Irr}(z)}{z} \right)^i \, \frac{{\bf Irr}(z)}{z}\\
&=\frac{1}{1-\frac{zu}{1-zu}\frac{u{\bf Irr}(z)}{z}}\,
\frac{zu{\bf Irr}(z)}{z(1-zu)^2}\\
&=\frac{uz^2({\bf S}(z)-1)}{(1-zu)^2-(1-zu)z^2u^2({\bf S}(z)-1)}.
\end{split}
\end{equation*}
It remains to consider RNA secondary structures that contain no irreducibles,
i.e.~RNA secondary structures consisting exclusively of isolated vertices.
Clearly,
\begin{equation}
{\bf V}(z,u)= \sum_{n \geq 1} z^n u^{n-1}=
\frac{z}{1-zu}
\end{equation}
and the proof of the theorem is complete.
\end{proof}

Setting ${\bf p}(n,d)={\bf w}(n,d)/{\bf w}(n)$, Theorem~\ref{T:exact}
provides the distribution of distances for RNA secondary structures of any
fixed length, $n$, see Tab.\ \ref{Tab:30}.

%%%%%%%%%%%%%%%%%%%%%%%%%%%%%%%%%%%%%%%%%%%%%%%%%%%%%%%%%%%%%%%%%%%

\section{The singular expansion}\label{S:singular}

%%%
%%%%%%%%%%%%%%%%%%%%%%%%%%%%%%%%%%%%%%%%%%%%%%%%%%%%%%%%%%%%%%%%%%%%%%
%%%

In this section we analyze the asymptotics of the $n$th coefficient,
$[z^n]{\bf W}(z,u)$. This will play a crucial role for the computation
of the limit distribution of distances in Section~\ref{S:limit}.

Let us first establish some facts needed for deriving the singular expansion:
%%%
%%%%%%%%%%%%%%%%%%%%%%%%%%%%%%%%%%%%%%%%%%%%%%%%%%%%%%%%%%%%%%%%%%%%%%%%%
%%%
\begin{lemma}
${\bf W}(z,u)$ is algebraic over the rational
function field $\mathbb{C}(z,u)$ and has the unique
dominant singularity, $\rho=(3-\sqrt{5})/2$,
which coincides with the unique dominant singularity
of ${\bf S}(z)$.
\end{lemma}
%%%
%%%%%%%%%%%%%%%%%%%%%%%%%%%%%%%%%%%%%%%%%%%%%%%%%%%%%%%%%%%%%%%%%%%%%%%%%
%%%
\begin{proof}
The fact that ${\bf W}(z,u)$ is algebraic over the rational
function $\mathbb{C}(z,u)$ follows immediately from
Theorem~\ref{T:exact} where we proved
\begin{equation*}
\begin{split}
&{\mathbf W}(z,u)=
\frac{uz^2 ({\mathbf S}(z)-1)}{(1-zu)^2-(1-zu)(zu)^2
({\mathbf S}(z)-1)}
+\frac{z}{1-zu},
\end{split}
\end{equation*}
since evidently all nominators and denominators are polynomial
expressions in $u$ and $z$ and
\begin{equation}\label{E:root}
{\bf S}(z)=\frac{1-z+z^2-\sqrt{(z^2 + z + 1) (z^2 - 3 z + 1)}}{2z^2}.
\end{equation}
Thus the field $\mathbb{C}(z,u)[{\bf S}(z)]$ is algebraic of degree two over
$\mathbb{C}(z,u)$.
The second assertion follows from $u\in (0,1)$ and a straightforward analysis
of the singularities of the two denominators $(1-zu)^2-(1-zu)(zu)^2
({\mathbf S}(z)-1)$ and $(1-zu)$.
\end{proof}

Given two numbers $\phi,r$, where $r>|\kappa|$ and $0<\phi<\frac{\pi}{2}$, the
open domain $\Delta_\kappa(\phi,r)$ is defined as
\begin{equation*}
\Delta_\kappa(\phi,r)=\{ z\mid \vert z\vert < r, z\neq \kappa,\,
\vert {\rm Arg}(z-\kappa)\vert >\phi\}.
\end{equation*}
A domain is a $\Delta_\kappa$-domain\index{$\Delta_\kappa$-domain} at
$\kappa$ if it
is of the form $\Delta_\kappa(\phi,r)$ for some $r$ and $\phi$. A
function is $\Delta_\kappa$-analytic\index{$\Delta_\kappa$-analytic} if
it is analytic in some $\Delta_\kappa$-domain.

Suppose an algebraic function has a unique singularity $\kappa$.
According to \citep{Flajolet:07a,Stanley:80}
such a function is $\Delta_\kappa(\phi,r)$-analytic.
In particular, ${\bf W}(z,u)$ is $\Delta_\rho(\phi,r)$-analytic.
We introduce the notation
\begin{eqnarray*}
\left(f(z)=o\left(g(z)\right) \
\text{\rm as $z\rightarrow \kappa$}\right)\  &\Longleftrightarrow& \
\left(f(z)/g(z)\rightarrow 0\ \text{\rm as $z\rightarrow \kappa$}\right),
\end{eqnarray*}
and if we write $f(z)=o\left(g(z)\right)$ it is implicitly
assumed that $z$ tends to the (unique) singularity.
The following transfer theorem allows us to obtain the asymptotics of
the coefficients from the generating functions.
%%%
%%%%%%%%%%%%%%%%%%%%%%%%%%%%%%%%%%%%%%%%%%%%%%%%%%%%%%%%%%%%%%%%%%%%%%%%%
%%%
\begin{theorem}\label{T:transfer1b}{\bf }\citep{Flajolet:07a}
{Let $f(z)$ be a $\Delta_{\kappa}$-analytic function at its unique singularity
$z=\kappa$. Let $g(z)\in \{(\kappa-z)^{\alpha}\mid \alpha \in \mathbb{R}\}$.
Suppose we have in the intersection of a neighborhood of $\kappa$ with the
$\Delta_{\kappa}$-domain
\begin{equation*}
f(z) =  o(g(z)) \quad \text{\it for } z\rightarrow \kappa.
\end{equation*}
Then we have
\begin{equation*}
[z^n]f(z)= o\left([z^n]g(z)\right).
\end{equation*}}
\end{theorem}
%%%
%%%%%%%%%%%%%%%%%%%%%%%%%%%%%%%%%%%%%%%%%%%%%%%%%%%%%%%%%%%%%%%%%%%%%%%%%
%%%
In addition, according to \citep{Flajolet:05} we have for
$\alpha\in\mathbb{C}\setminus \mathbb{Z}_{\le 0}$:
\begin{eqnarray}\label{E:33}
[z^n]\, (1-z)^{-\alpha} & \sim & \frac{n^{\alpha-1}}{\Gamma(\alpha)}\left[
1+\frac{\alpha(\alpha-1)}{2n}+ O\left(\frac{1}{n^2}\right)\right].
\end{eqnarray}
We next observe ${\bf W}(z,u)=h(z,u)\, f(g(z,u))$,
where $g(z,u)=(uz^2({\bf S}(z)-1))/(1-uz)$, $f(z)=z/(1-uz)$,
$h(z,u)=1/(1-zu)$ and $t(z,u)=uz^2/(1-uz)$.
In preparation for the proof of Lemma~\ref{L:erni1}
we set
\begin{equation*}
\begin{split}
\alpha &=  g(\rho,u)=\frac{2(-2+\sqrt{5})u}{2+(-3+\sqrt{5})u}\\
C_0 &= \frac{2}{2-(3-\sqrt{5})u}\,\\
&\left(f(\alpha)+\frac{d f(w)}{d w}|_{w=\alpha}\,t(\rho,u)
\,\frac{\sqrt{5}-1}{3-\sqrt{5}}-\alpha\,\frac{d f(w)}{d w}|_{w=\alpha}
+\rho\right)\\
r(\rho,u) &= - \frac{2}{2-(3-\sqrt{5})u}\,\frac{d f(w)}{d w}|_{w=\alpha}\,
t(\rho,u)\,\frac{\sqrt{8(3\sqrt{5}-5)}}{(-3+\sqrt{5})^2}.
\end{split}
\end{equation*}
Furthermore, let $v(z)$ and $w(z)$ be $D$-finite power series such that $w(0)=0$
and let $\rho_v$, $\rho_w$ denote their respective radius of convergence.
We set $\tau_w =\lim_{z \rightarrow \rho_w^{-}} w(z)$ and call
the $D$-finite power series $F(z) = v(w(z))$ subcritical if and only if
$\tau_w < \rho_v$.

%%%
%%%%%%%%%%%%%%%%%%%%%%%%%%%%%%%%%%%%%%%%%%%%%%%%%%%%%%%%%%%%%%%%%%%%%%%%%
%%%
\begin{lemma}\label{L:erni1}
The singular expansion of ${\bf W}(z,u)$ at its unique, dominant singularity
$\rho$ is given by
\begin{equation}
{\bf W}(z,u)=C_0+{\bf V}(\rho,u)+ r(\rho,u)(\rho-z)^{1/2}+O(\rho-z).
\end{equation}
\end{lemma}
%%%
%%%%%%%%%%%%%%%%%%%%%%%%%%%%%%%%%%%%%%%%%%%%%%%%%%%%%%%%%%%%%%%%%%%%%%%%%
%%%
\begin{proof}
Since $g(0,u)=0$, the composition $f(g(z,u))$ is well defined as a formal
power series and ${\bf V}(z,u)=\frac{z}{1-zu}$ as well as $h(z,u)$ are
regular at $\rho$.
Since $u \in (0,1)$ we have $1/u> 1> \rho$, whence the dominant
singularity
of $g(z,u)$ equals $\rho$. Next we observe
\begin{equation*}
g(\rho,u)=\frac{u(1-\rho-\rho^2)}{2(1-u\rho)}<
               \frac{0.7 u}{2(1-0.4u)}=\frac{0.35u}{1-0.4u}<1,
\end{equation*}
whence $f(g(z,u))$ is governed by the subcritical paradigm. \\
{\it Claim $1$.}
\begin{equation}
g(z,u)=t(\rho,u)\,\frac{2}{3-\sqrt{5}}- t(\rho,u)\,
\frac{\sqrt{8(3\sqrt{5}-5)(\rho-z)}}{(-3+\sqrt{5})^2} +O(\rho-z).
\end{equation}
To prove the Claim we consider the singular expansion of ${\bf S}(z)$ at
$\rho$
\begin{equation}
{\bf S}(z)=\frac{2}{3-\sqrt{5}}-\frac{\sqrt{8(3\sqrt{5}-5)
(\rho-z)}}{(-3+\sqrt{5})^2}+O(\rho-z).
\end{equation}
The singular expansion
of $g(z,u)$ at $\rho$ is obtained by multiplying the regular expansion of
$t(z,u)$ and singular expansion of ${\bf S}(z)-1$. Clearly,
\begin{equation}
t(z,u)=t(\rho,u)-\frac{d t(z,u)}{dz}|_{z=\rho} \,(\rho-z)+O((\rho-z)^2),
\end{equation}
where $t(\rho,u)=(7-3\sqrt{5})u/(2-(3-\sqrt{5})u)$. Thus
\begin{equation}
g(z,u)=t(\rho,u)\,\frac{\sqrt{5}-1}{3-\sqrt{5}}- t(\rho,u)\,
\frac{\sqrt{8(3\sqrt{5}-5)(\rho-z)}}{(-3+\sqrt{5})^2}+O(\rho-z).
\end{equation}
Setting $\alpha=g(\rho,u)=2(-2+\sqrt{5})u/(2+(-3+\sqrt{5})u)$,
the regular expansion of $f(w)$ at $\alpha$ is
\begin{equation}
f(w) =
f(\alpha)+\frac{d f(w)}{dw}|_{w=\alpha}\,(w-\alpha)-O(w-\alpha),
\end{equation}
where $\frac{d f(w)}{dw}|_{w=\alpha}=
\left(\frac{2+(-3+\sqrt{5})u}{2+(-3+\sqrt{5})u-
2(-2+\sqrt{5})u^2}\right)^2$,
and accordingly
\begin{equation}
f(g(z,u)) = C_1-
\frac{d f(w)}{dw}|_{w=\alpha}\,t(\rho,u)\,\frac{\sqrt{8(3\sqrt{5}-5)(\rho-z)}}
{(-3+\sqrt{5})^2}+O(\rho-z),
\end{equation}
where
$C_1=f(\alpha)+\frac{d f(w)}{dw}|_{w=\alpha}\,t(\rho,u)\,\frac{\sqrt{5}-1}{3-
\sqrt{5}}-\alpha\,\frac{d f(w)}{dw}|_{w=\alpha}$.
Multiplying by the regular expansion of $h(z,u)$ at $\rho$ and adding
the regular expansion of ${\bf V}(z,u)$ implies the lemma.
\end{proof}

\section{The limit distribution}\label{S:limit}

In this Section we shall prove that for any finite $d$ holds
\begin{equation}
\lim_{n\to\infty}\frac{{\bf w}(n,d)}{{\bf w}(n)}={\bf q}(d).
\end{equation}
We furthermore determine the limit distribution via computing the
power series
\begin{equation}
{\bf Q}(u)=\sum_{d\ge 1}{\bf q}(d)u^d.
\end{equation}
Theorem~\ref{T:continuity} below ensures that under certain conditions the
point-wise convergence of probability generating functions implies
the convergence of its coefficients.
%%%
%%%%%%%%%%%%%%%%%%%%%%%%%%%%%%%%%%%%%%%%%%%%%%%%%%%%%%%%%%%%%%%%%%%%%%%%%
%%%
\begin{theorem}{\label{T:continuity}}
Let $u$ be an indeterminate and $\Omega$ be a set contained in the
unit disc, having at least one accumulation point in the interior of
the disc. Assume ${\bf P}_n(u)=\sum_{d\ge 0}{\bf p}(n,d)u^d$ and
${\bf Q}(u)=\sum_{d\ge 0}{\bf q}(d) u^k$ such that\\
$\lim_{n\rightarrow \infty}{\bf P}_n(u)={\bf Q}(u)$ for each $u\in\Omega$ holds.
Then we have for any finite $d$,
\begin{equation}
\lim_{n\rightarrow\infty}{\bf p}(n,d)={\bf q}(d) \quad \ \text{\it and }\quad \
\lim_{n\rightarrow \infty}\sum_{j\le d}{\bf p}(n,j)=\sum_{j\le d}{\bf q}(j).
\end{equation}
\end{theorem}
%%%
%%%%%%%%%%%%%%%%%%%%%%%%%%%%%%%%%%%%%%%%%%%%%%%%%%%%%%%%%%%%%%%%%%%%%%%%%
%%%
Let ${m}_1(u)=(-7 + 3 \sqrt{5}) u$ and
$$
{m}_2(u)=-2 -2 (-3 + \sqrt{5}) u + (-15 + 7 \sqrt{5}) u^2 +
(22 - 10 \sqrt{5}) u^3 +2 (-9 + 4 \sqrt{5}) u^4.
$$

%%%
%%%%%%%%%%%%%%%%%%%%%%%%%%%%%%%%%%%%%%%%%%%%%%%%%%%%%%%%%%%%%%%%%%%%%%%%%
%%%
\begin{theorem}\label{T:limit}
For any $d\ge 1$ holds
\begin{equation}
\lim_{n\to\infty}{\bf p}(n,d)
=\lim_{n\to\infty}\frac{{\bf w}(n,d)}{{\bf w}(n)}={\bf q}(d),
\end{equation}
where ${\bf q}(d)$ is given via the probability generating
function ${\bf Q}(u)$
\begin{equation}
{\bf Q}(u)=\frac{{m}_1(u)}{{m}_2(u)}.
\end{equation}
\end{theorem}
%%%
%%%%%%%%%%%%%%%%%%%%%%%%%%%%%%%%%%%%%%%%%%%%%%%%%%%%%%%%%%%%%%%%%%%%%%%%%
%%%
\begin{proof}
According to Lemma~\ref{L:erni1}, the singular expansion of
${\bf W}(z,u)$ is given by
\begin{equation}
{\bf W}(z,u)=C_0 + {\bf V}(\rho,u)+ r(\rho,u)(\rho-z)^{1/2}+O(\rho-z).
\end{equation}
Thus
\begin{equation}
[z^n]{\bf W}(z,u)=r(\rho,u)\, [z^n]\,(\rho-z)^{1/2} + [z^n]\,O(\rho-z).
\end{equation}
In view of $O(z-\rho)=o((z-\rho)^{1/2})$, Theorem~\ref{T:transfer1b}
implies
\begin{equation}
[z^n]{\bf W}(z,u)\sim r(\rho,u)\,[z^n]\,(\rho-z)^{1/2}.
\end{equation}
Employing eq.~(\ref{E:33}) we obtain
\begin{equation}
[z^n]{\bf W}(z,u)\sim r(\rho,u) \, K \, n^{-3/2}\,\rho^{-n}(1+O(\frac{1}{n})),
\end{equation}
for some constant $K>0$. Substituting for $r(\rho,u)$ we arrive at
\begin{equation*}
[z^n]{\bf W}(z,u)=
\frac{{m}_1(u)}{{ m}_2(u) } \cdot \frac{2\sqrt{6\sqrt{5}-10}}{(-3+\sqrt{5})^2} \cdot
K \,  n^{-3/2}\rho^{-n}(1+O(\frac{1}{n}))
\end{equation*}
and in particular for $u=1$
\begin{equation*}
[z^n]{\bf W}(z,1)=\frac{2\sqrt{6\sqrt{5}-10}}{(-3+\sqrt{5})^2} \cdot
K\,  n^{-3/2}\, \rho^{-n}(1+O(\frac{1}{n})).
\end{equation*}
We consequently have
\begin{equation}\label{E:converge}
\begin{split}
\lim_{n \rightarrow \infty}\frac{[z^n]{\bf W}(z,u)}{[z^n]{\bf W}(z,1)}
=\frac{{m}_1(u)}{{m}_2(u)}.
\end{split}
\end{equation}
Therefore, setting ${\bf P}_n(u)=\sum_d{\bf p}(n,d)u^d$,
\begin{equation}
\lim_{n\to\infty}{\bf P}_n(u)={\bf Q}(u).
\end{equation}
Since $u \in (0,1)$, $0$ is an accumulation point of $\Omega=(0,1)$,
and eq.~(\ref{E:converge}) holds for each $u \in \Omega$,
Theorem~\ref{T:continuity} implies for any finite $d$
\begin{equation}
\lim_{n\to \infty} {\bf p}(n,d)=\lim_{n\to\infty}\frac{{\bf w}(n,d)}{{\bf w}(n)}=
{\bf q}(d).
\end{equation}
\end{proof}

We finally compute the asymptotic expression of ${\bf q}(d)$. For this purpose
we recall that the density function of a $\Gamma{(\lambda, r)}$-distribution is
given by
\begin{equation}
f_{\lambda, r}(x)=
\begin{cases}
\frac{\lambda^r}{\Gamma{(\lambda)}}\,x^{r-1} e^{-\lambda x}, \quad x>0 \\
0, \quad x \geq 0
\end{cases}
\end{equation}
where $\lambda >0$ and $r>0$.

%%%
%%%%%%%%%%%%%%%%%%%%%%%%%%%%%%%%%%%%%%%%%%%%%%%%%%%%%%%%%%%%%%%%%%%%%%%%%
%%%
\begin{corollary}\label{C:1}
Let $\rho$ be the real positive dominant singularity of ${\bf S}(z)$
and set $\delta=\frac{1}{4}(-1 - \sqrt{5} + \sqrt{38+18\sqrt{5}})$. Then
\begin{equation*}
 {\bf q}(d)\sim \frac{C_3}{\delta}(d+1)(\frac{1}{\delta})^{d+1}
 =\frac{C_3}{\delta}\,(\ln(\delta))^{-2}\, f_{\ln(\delta),2} (d).
\end{equation*}
That is, in the limit of large distances the coefficient ${\bf q}(d)$
is determined by the density function of a $\Gamma(\ln \delta, 2)$-distribution.
\end{corollary}
%%%
%%%%%%%%%%%%%%%%%%%%%%%%%%%%%%%%%%%%%%%%%%%%%%%%%%%%%%%%%%%%%%%%%%%%%%%%%
%%%

%%%
%%%%%%%%%%%%%%%%%%%%%%%%%%%%%%%%%%%%%%%%%%%%%%%%%%%%%%%%%%%%%%%%%%%%%%%%%
%%%
\section{Discussion}\label{S:discuss}
%%%
%%%%%%%%%%%%%%%%%%%%%%%%%%%%%%%%%%%%%%%%%%%%%%%%%%%%%%%%%%%%%%%%%%%%%%%%%
%%%

The results of this paper suggest that the number of base pairs alone is not
sufficient to explain the distribution of $5'$-$3'$ distances. Surprisingly,
we find that the $5'$-$3'$ distances of random are much smaller than those of
mfe-structures, despite the fact that they contain a lesser number of base
pairs, see Fig.\ \ref{F:kkk}.
%%%%%%%%%%%%%%%%%%%%%%%%%%%%%%%%%%%%%%%%%%%%%%%%%%%%%%%%%%%%%%%%%%%%%%%%%%
%%%

By definition, only irreducibles and isolated vertices contribute to the
$5'$-$3'$ distance. The particular number of base pairs contained within
irreducible substructures is irrelevant.
It has been shown in \citep{Emma:09} that there exists a limit distribution
for the number of irreducibles in random RNA secondary structures. This
limit distribution is a determined by a $\Gamma$-distribution similar to
Corollary~\ref{C:1}.
As a result, random RNA secondary structures have only very few irreducibles,
typically two or three. This constitutes a feature shared by RNA mfe-structures.
Thus in case of random and mfe-structures a few irreducibles ``cover'' almost
the entire sequence since the $5'$-$3'$ distance is, even in the limit of large
sequence length, finite. The distinctively larger $5'$-$3'$ distance of
mfe-structures consequently stems from the fact that their irreducibles cover
a distinctively smaller fraction of the sequence.
Hence the irreducibles of mfe-structures differ in a subtle way from those of
random RNA structures. We show in the following that the shift of the
$5'$-$3'$ distance is a combinatorial consequence of large stacks observed
in mfe-structures, see Fig.~\ref{F:kk1}.
%%%%%%%%%%%%%%%%%%%%%%%%%%%%%%%%%%%%%%%%%%%%%%%%%%%%%%%%%%%%%%%%%%%%%%%%%%
%%%

Here a stack of length $r$ is a maximal sequence of
``parallel'' arcs, $((i,j),(i+1,j-1),\dots,(i+(r-1),j-(r-1)))$.
RNA secondary structures with stack length $\ge r$ is called
$r$-canonical RNA secondary structures.
Let ${\mathbf w}_r(n,d)$ denote the number of $r$-canonical RNA
secondary structures $\sigma_{r,n}$ having distance $d_n$.
We shall write $d$ instead of $d_n$ and consider
\begin{equation}
{\mathbf W}_r(z,u)=\sum_{n\geq 0}\sum_{d \geq 0}{\bf w}_r(n,d)\,z^n u^d,
\end{equation}
the bivariate generating function of the number of RNA secondary
structure with minimum stack-size $r$ of length $n$ having distance $d$
and set ${\mathbf w}_r(n)=\sum_{d \geq 0} {\mathbf w}_r(n,d)$. Let
${\mathbf S}_{r}(z)$ denote the generating function of $r$-canonical
RNA secondary structures. Set
\begin{equation}\label{E:P}
\begin{split}
p_r(z)=&(z^{2r}-(z-1)(z^{2r}-z^2+1))^2-4z^{2r}(z^{2r}-z^2+1).
\end{split}
\end{equation}
Then the generating function of $r$-canonical secondary structures is given
by
\begin{equation}\label{E:root1}
{\bf S}_r(z) =  \frac{(z^{2r}-(z-1)(z^{2r}-z^2+1)-\sqrt{p_r(z)})}{2z^{2r}}.
\end{equation}
and we can derive it using symbolic enumeration \citep{Flajolet:07a}.

%%%%%%%%%%%%%%%%%%%%%%%%%%%%%%%%%%%%%%%%%%%%%%%%%%%%%%%%%%%%%%%%%%%%%%%
\begin{theorem}\label{T:exact2}
The bivariate generating function of the number of $r$-canonical
RNA secondary structures of length $n$ with distance $d$, is given by
\begin{equation}
\begin{split}
&{\mathbf W}_r(z,u)=
\frac{u\, z^{2r} ({\mathbf S}_r(z)-1)}{(1-zu)^2\,(1-z^2+z^{2r})
-(1-zu)u^2\, z^{2r}
({\mathbf S}_r(z)-1)}
+\frac{z}{1-zu}.
\end{split}
\end{equation}
\end{theorem}
%%%%%%%%%%%%%%%%%%%%%%%%%%%%%%%%%%%%%%%%%%%%%%%%%%%%%%%%%%%%%%%%%%%%%%%%%
Along the lines of our analysis subsequent to Theorem~\ref{T:exact}
we can then obtain the singular expansion and the limit distributions
for the $5'$-$3'$ distances of $r$-canonical RNA secondary structures,
see Fig.~\ref{F:kkk}.

\section{ Acknowledgments.}
%%%
%%%%%%%%%%%%%%%%%%%%%%%%%%%%%%%%%%%%%%%%%%%%%%%%%%%%%%%%%%%%%%%%%%%%%%%%%%
%%%
We are grateful to Thomas J. X. Li for carefully reading the manuscript
and special thanks for Fenix W.D.~Huang for generating Figure.~\ref{F:kk1}
and Emeric Deutsch for pointing out an error in Theorem.~\ref{T:exact2}
in the discussion.

%%%%%%%%%%%%%%%%%%%%%%%%%%%%%%%%%%%%%%%%%%%%%%%%%%%%%%%%%%%%%%%
%%%%%%%%%%%%%%%%%%%%%%%%%%%%%%%%%%%%%%%%%%%%%%%%%%%%%%%%%%%%%%%%%%%



%% References
%% Usage of \citep is as follows:
%%   \citep{key}         ==>>  [#]
%%   \citep[chap. 2]{key} ==>> [#, chap. 2]
%%


%\bibliographystyle{elsarticle-num}






% \end{thebibliography}

%%%%%%%%%%%%%%%%%%%%%%%%%%%%%%%%%%%%%%%%%%%%%%%%%%%%%%%%%%%%%%%%%%%%%%%%%%%%%%
%%%

%\def\papername{Bibliography}\label{bibliographypapge}
%\bibliography{Sdistance}

\begin{thebibliography}{00}
\doublespacing

\bibitem[Chen {\em et~al.}(2008)]{Reidys:vac07}
Chen, W. Y. C., Qin, J., Reidys, C. M. (2008).
\newblock Crossings and nestings of tangled diagrams.
\newblock Elec. J. Comb. {\bf 15}, {\#}86.

\bibitem[Chen {\em et~al.}(2007)]{Chen}
Chen, W. Y. C., Deng, E. Y. P., Du, R. R. X., et al. (2007).
\newblock Crossings and nestings of matchings and partitions.
\newblock Trans. Amer. Math. Soc. {\bf 359(4)},  1555--1575.

\bibitem[Jin {\em et~al.}(2008)]{Reidys:07pseu}
Jin, E. Y., Qin, J., Reidys, C. M. (2008).
\newblock Combinatorics of {RNA} structures with pseudoknots.
\newblock Bull. Math. Biol. {\bf 70}, 45--67.

\bibitem[Jin and Reidys(2010a)]{Emma:09}
Jin, E. Y., Reidys, C. M. (2010a).
\newblock Irreducibility in RNA structures.
\newblock Bull. Math. Biol. {\bf 72},
 375--399.

\bibitem[Jin and Reidys(2010b)]{Emma:decom}
Jin, E. Y., Reidys, C. M. (2010b).
\newblock On the decomposition of k-noncrossing {RNA} structures.
\newblock Adv. Appl. Math. {\bf 44(1)}, 53--70.


\bibitem[Flajolet {\em et~al.}(2005)]{Flajolet:05}
Flajolet, P., Fill, J. A., Kapur, N. (2005).
\newblock Singularity analysis, hadamard products, and
  tree recurrences.
\newblock J. Comp. Appl. Math. {\bf 174}, 271--313.

\bibitem[Flajolet and Sedgewick(2009)]{Flajolet:07a}
Flajolet, P., Sedgewick, R. (2009).
\newblock Analytic Combinatorics. Cambridge University Press, New York.

\bibitem[Fontana {\em et~al.}(1993)]{Schuster:93}
Fontana, W., Konings, D. A. M., Stadler, P. F. (1993).
\newblock Statistics of {RNA} secondary structures.
\newblock Biopolymers {\bf 33(9)}, 1389--1404.


\bibitem[Howell {\em et~al.}(1980)]{Waterman:80}
Howell, J. A., Smith, T. F., Waterman, M. S. (1980).
\newblock Computation of generating functions for
biological molecules.
\newblock SIAM J. Appl. Math. {\bf 39}, 119--133.


\bibitem[Penner and Waterman(1993)]{Penner:93c}
Penner, R. C., Waterman, M. S. (1993).
\newblock Spaces of {RNA} secondary structures.
\newblock Adv. Math. {\bf 101},
31--49.

\bibitem[Reidys(2011)]{Reidys:book}
Reidys, C. M. (2011).
\newblock Combinatorial Computational Biology of {RNA}.
\newblock Springer-Verlag, New York.


 \bibitem[Stanley(1980)]{Stanley:80}
Stanley, R. P. (1980).
\newblock Differentiably finite power series. Europ.
\newblock J. Comb. {\bf 1},
  175--188.


\bibitem[Waterman(1979)]{Waterman:79a}
Waterman, M. S. (1979).
\newblock Combinatorics of {RNA} hairpins and cloverleafs.
\newblock Stud. Appl. Math.
  {\bf 60}, 91--96.


\bibitem[Waterman(1978)]{Waterman:78a}
Waterman, M. S. (1978).
\newblock Secondary structure of single-stranded nucleic acids.
\newblock Adv. Math. (Suppl. Studies) {\bf 1}, 167--212.


\bibitem[Waterman and Schmitt(1994)]{Waterman:94a}
Waterman, M. S., Schmitt, W. R. (1994).
\newblock Linear trees and {RNA} secondary structure.
\newblock Discr. Appl. Math. {\bf 51}, 317--323.

\bibitem[Yoffe {\em et~al.}(2011)]{Yoffe}
Yoffe, A. M., Prinsen, P., Gelbart, W. M., et al. (2011).
\newblock The ends of a large {RNA} molecule are necessarily close.
\newblock Nucl. Acid. Res. {\bf 39(1)}, 292--299.

\end{thebibliography}
%\bibliographystyle{elsarticle-num}
%\newpage

\newpage
%%%%%%%%%%%%%%%%%%%%%%%%%%%%%%%%%%%%%%%%%%%%%%%%%%%%%%%%%%%%%%%%%%%%
%%%%
\begin{table}[ht]
\renewcommand{\arraystretch}{0.5}
\setlength{\tabcolsep}{1pt}
\begin{center}
\def\temptablewidth{2\textwidth}
\begin{tabular}{|c|c|c|c|c|c|c|c|}
\hline
\hline
$d$ & \small{1} &\small{2} &\small{3} &\small{4} &\small{5} & \small{6}    \\
\small{$\mathbf{p}(n,d)$}  & \small{0.161} & \small{0.129} & \small{0.148} &  \small{0.126} & \small{0.109} & \small{ 0.088}  \\
\hline
$d$  & \small{7} & \small{8} & \small{9} &\small{10} & \small{11} &\small{12}   \\
\small{${\bf p}(n,d)$} & \small{0.069} & \small{$5.18\times10^{-2}$} & \small{$3.8\times10^{-2}$} & \small{$2.71\times10^{-2}$} & \small{$1.87\times10^{-2}$} & \small{$1.26\times10^{-2}$}  \\
\hline
$d$  &\small{13} & \small{14} & \small{15} &\small{16} & \small{17} & \small{18} \\
\small{${\bf p}(n,d)$}   & \small{$8.22\times10^{-3}$} & \small{$5.19\times10^{-3}$} & \small{$3.17\times10^{-3}$} &
\small{$1.86\times10^{-3}$} & \small{$1.05\times10^{-3}$}& \small{$5.62\times10^{-4}$}  \\
\hline
$d$  & \small{19} & \small{20}  &\small{21} & \small{22} & \small{23} &\small{24}  \\
\small{${\bf p}(n,d)$} & \small{$2.85\times10^{-4}$} & \small{$1.36\times10^{-4}$}
& \small{$5.99\times10^{-5}$} &
\small{$2.41\times10^{-5}$} & \small{$8.58\times10^{-6}$} & \small{$2.63\times10^{-6}$}\\
\hline
 $d$   & \small{25} & \small{26} & \small{27} & \small{28} & \small{29} & \\
\small{${\bf p}(n,d)$}  & \small{$6.56\times10^{-7}$} & \small{$1.24\times10^{-7}$} & \small{$1.64\times10^{-8}$} & \small{$1.30\times10^{-9}$} & \small{$4.65\times10^{-11}$} & \\
\hline
\hline
\end{tabular}
\end{center}
\vspace*{1pt}
\caption{\small The distribution of distances of RNA secondary
structures of length $30$. The data of this table are represented
in Fig.\ \ref{F:curve} as ``$+$''.}
\label{Tab:30}
\end{table}

%%%%%%%%%%%%%%%%%%%%%%%%%%%%%%%%%%%%%%%%%%%%%%%%%%%%%%%%%%%%%%%%%%

\newpage
%%%
%%%%%%%%%%%%%%%%%%%%%%%%%%%%%%%%%%%%%%%%%%%%%%%%%%%%%%%%%%%%%%%%%%%%%%%%%%
%%%
\begin{figure}[ht]
\centerline{\epsfig{file=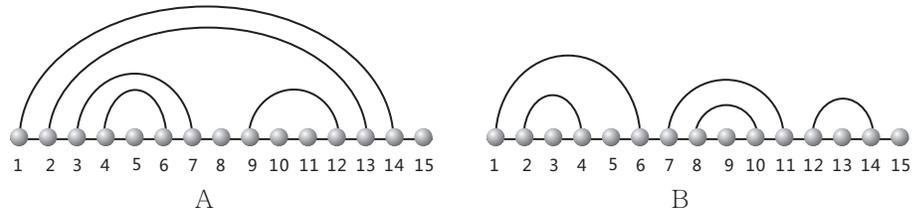,width=0.8\textwidth}
\hskip8pt}
\caption{\small RNA secondary structures as diagrams:
the backbone of the RNA molecule is drawn as a horizontal line and
Watson-Crick base pairs are represented as arcs in the upper half-plane.
An RNA secondary structure has no $1$-arcs and only noncrossing arcs.
}\label{F:secon}
\end{figure}
%%%
%%%%%%%%%%%%%%%%%%%%%%%%%%%%%%%%%%%%%%%%%%%%%%%%%%%%%%%%%%%%%%%%%%%%%%%%%%
%%%

%%%%%%%%%%%%%%%%%%%%%%%%%%%%%%%%%%%%%%%%%%%%%%%%%%%%%%%%%%%%%%%%%%%%%%%%%%
%%%
\begin{figure}[ht]
\centerline{\epsfig{file=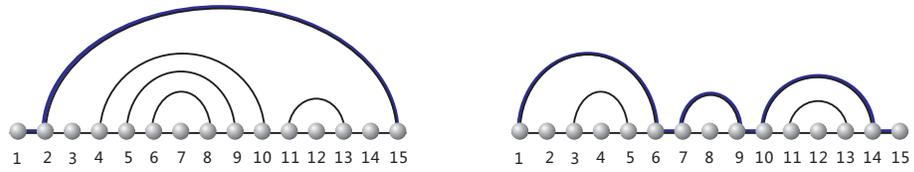,width=0.8\textwidth}
\hskip8pt}
\caption{\small The $5'$-$3'$ distance of RNA secondary structures:
distance contributing backbone-edges and arcs are drawn in blue.
The structure on the lhs has $5'$-$3'$ distance $2$ and structure on the rhs has
$5'$-$3'$ distance $6$.
}\label{F:distance0}
\end{figure}
%%%
%%%%%%%%%%%%%%%%%%%%%%%%%%%%%%%%%%%%%%%%%%%%%%%%%%%%%%%%%%%%%%%%%%%%%%%%%%
%%%
%%%
%%%%%%%%%%%%%%%%%%%%%%%%%%%%%%%%%%%%%%%%%%%%%%%%%%%%%%%%%%%%%%%%%%%%%%%%%%
%%%
\begin{figure}[ht]
\centerline{%
\epsfig{file=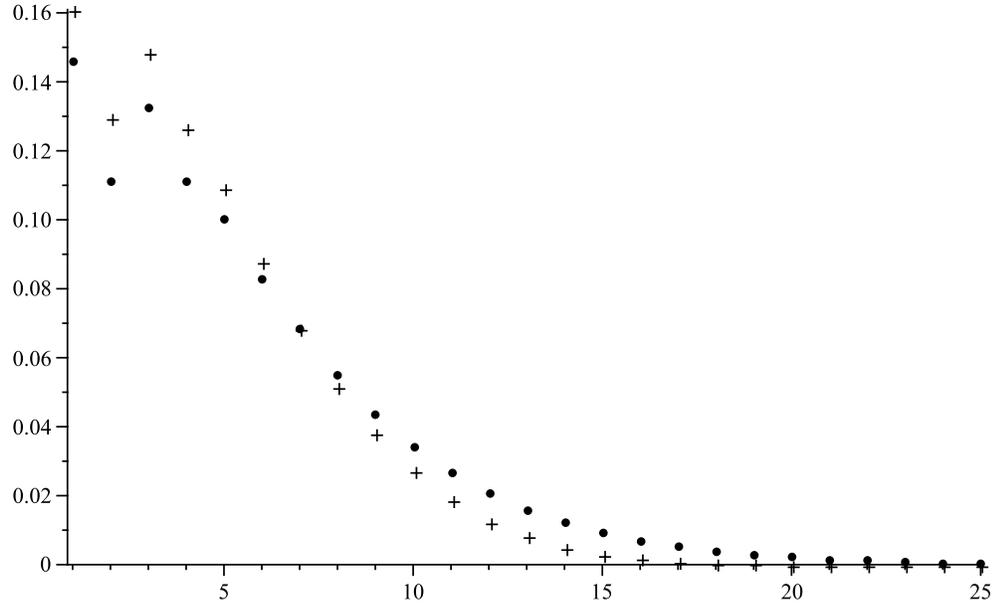,width=0.9\textwidth}\hskip15pt }
\caption{\small The distribution of $5'$-$3'$ distances of RNA secondary
structures: We display the distribution of distances in RNA secondary
structures of length $30$ ($+$) derived via Theorem~\ref{T:exact}.
We furthermore show the distribution of distances in the limit of
long RNA secondary structures ($\bullet$) obtained via
Theorem~\ref{T:limit}.} \label{F:curve}
\end{figure}
%%%
%%%%%%%%%%%%%%%%%%%%%%%%%%%%%%%%%%%%%%%%%%%%%%%%%%%%%%%%%%%%%%%%%%%%%%%%%%
%%%
%%%%%%%%%%%%%%%%%%%%%%%%%%%%%%%%%%%%%%%%%%%%%%%%%%%%%%%%%%%%%%%%%%%%%%%%%%
%%%
\begin{figure}[ht]
\centerline{%
\epsfig{file=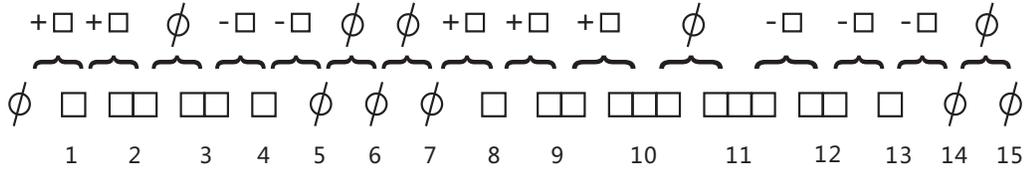,width=0.9\textwidth}\hskip15pt }
\caption{\small A $1$-tableaux: at each step either nothing happens or
a single $\Box$ is added or removed.} \label{F:ttt}
\end{figure}
%%%
%%%%%%%%%%%%%%%%%%%%%%%%%%%%%%%%%%%%%%%%%%%%%%%%%%%%%%%%%%%%%%%%%%%%%%%%%%
%%%

%%%%%%%%%%%%%%%%%%%%%%%%%%%%%%%%%%%%%%%%%%%%%%%%%%%%%%%%%%%%%%%%%%%%%%%%%%
%%%
\begin{figure}[ht]
\centerline{%
\epsfig{file=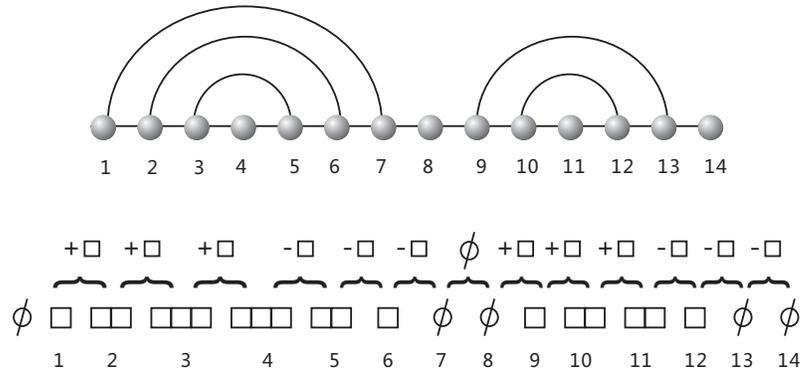,width=0.7\textwidth}\hskip15pt }
\caption{\small Mapping RNA secondary structures into
$1$-tableaux.} \label{F:bijection}
\end{figure}
%%%
%%%%%%%%%%%%%%%%%%%%%%%%%%%%%%%%%%%%%%%%%%%%%%%%%%%%%%%%%%%%%%%%%%%%%%%%%%

%%%%%%%%%%%%%%%%%%%%%%%%%%%%%%%%%%%%%%%%%%%%%%%%%%%%%%%%%%%%%%%%%%%%%%%%%%
%%%
\begin{figure}[ht]
\centerline{%
\epsfig{file=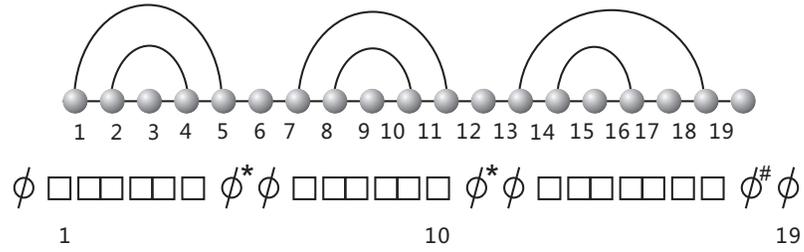,width=0.7\textwidth}\hskip15pt }
\caption{\small A secondary structure and its a $1$-tableaux:
its $5'$-$3'$ distance equals twice the number of $\varnothing^{*}$
plus the number of $\varnothing$ shapes, i.e.~$2\times 2+4=8$.} \label{F:curve2}
\end{figure}
%%%
%%%%%%%%%%%%%%%%%%%%%%%%%%%%%%%%%%%%%%%%%%%%%%%%%%%%%%%%%%%%%%%%%%%%%%%%%%
%%%
%%%%%%%%%%%%%%%%%%%%%%%%%%%%%%%%%%%%%%%%%%%%%%%%%%%%%%%%%%%%%%%%%%%%%%%%%%
%%%
\begin{figure}[ht]
\centerline{%
\epsfig{file=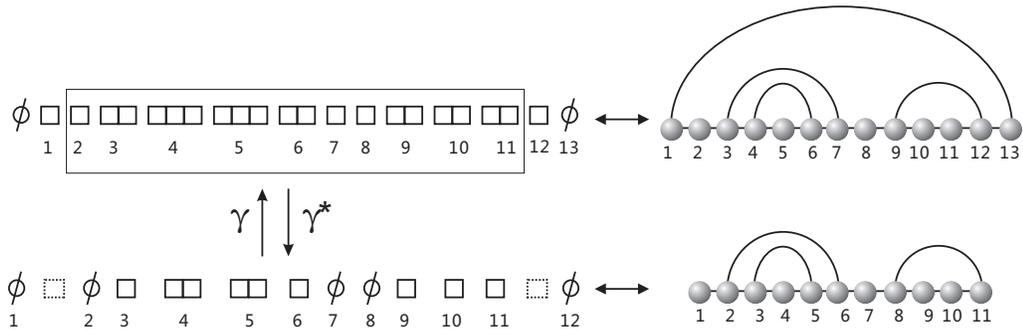,width=0.9\textwidth}\hskip15pt }
\caption{\small The mappings $\gamma$ and $\gamma^*$.} \label{F:mapIS}
\end{figure}
%%%
%%%%%%%%%%%%%%%%%%%%%%%%%%%%%%%%%%%%%%%%%%%%%%%%%%%%%%%%%%%%%%%%%%%%%%%%%%

\begin{figure}[ht]
\centerline{%
\epsfig{file=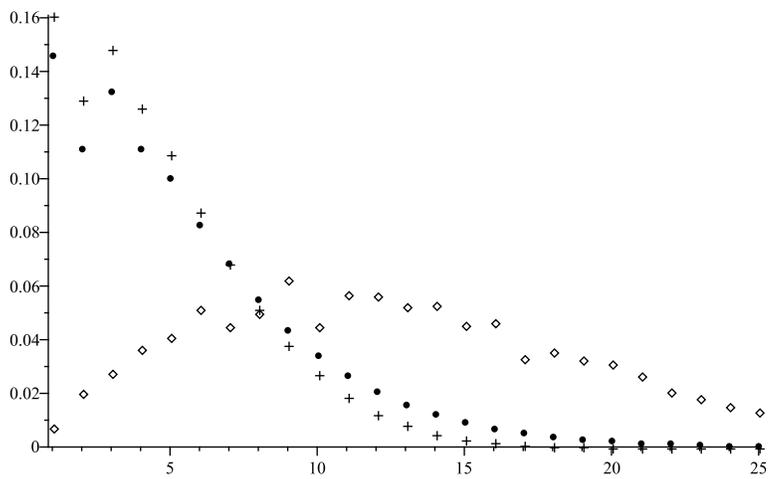,width=0.7\textwidth}\hskip15pt }
\caption{\small The $5'$-$3'$ distance of random structures and mfe-structures:
We display RNA secondary structures of length $30$ ($+$) and the limit
distribution ($\bullet$) as well as a sample of $5000$ mfe-structures
obtained from random sequences of length $100$ ($\diamond$).} \label{F:kk1}
\end{figure}
%%%
%%%%%%%%%%%%%%%%%%%%%%%%%%%%%%%%%%%%%%%%%%%%%%%%%%%%%%%%%%%%%%%%%%%%%%%%%%
%%%
\begin{figure}[ht]
\centerline{%
\epsfig{file=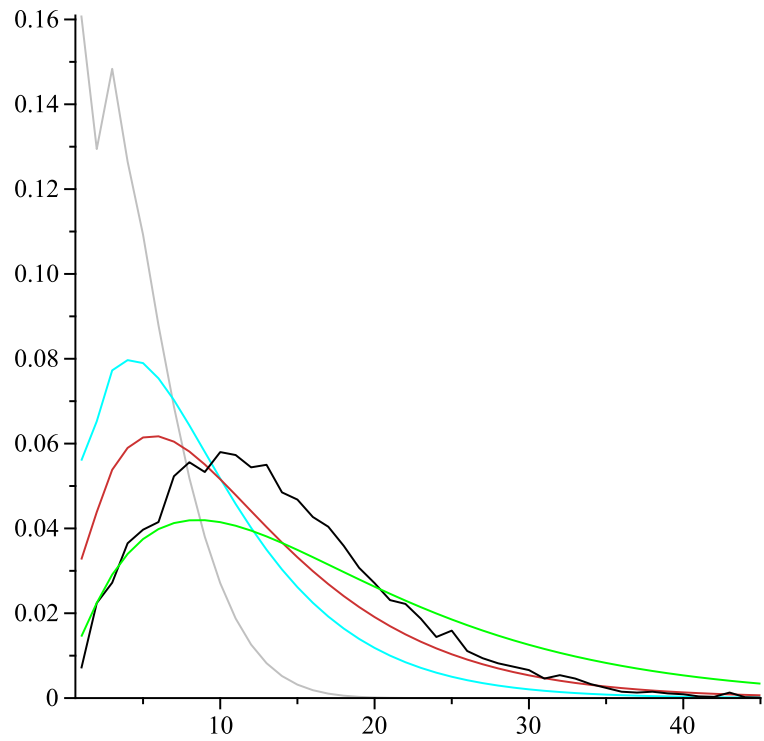,width=0.7\textwidth}\hskip15pt }
\caption{\small The $5'$-$3'$ limit distance distribution of $r$-canonical
RNA structures and mfe-structures:
We display limit distance distribution of $r$-canonical RNA structures
of length $45$: (gray line: $r=1$), (cyan line: r=3), (orange line: r=5),
(green line: r=10)
as well as a sample of $10000$ mfe-structures
obtained from random sequences of length $100$ (black line).} \label{F:kkk}
\end{figure}
%%%
%%%%%%%%%%%%%%%%%%%%%%%%%%%%%%%%%%%%%%%%%%%%%%%%%%%%%%%%%%%%%%%%%%%%%%%%%%
%%%

\end{document}